%% file: Chev-Eil_quant.tex
\documentclass{amsart}

\usepackage{latexsym,amsmath,amstext,mathrsfs,amsthm,amscd,amsopn,verbatim,amsrefs}

\usepackage{graphicx}
\usepackage[dvips]{color}
\usepackage{epsfig}
\usepackage{psfrag} 
\usepackage[all]{xy}
\usepackage{amscd}
\usepackage{amssymb}
\usepackage{MnSymbol}

\newtheorem{Thm}{Theorem}[section]
\newtheorem{Prop}[Thm]{Proposition}
\newtheorem{Lem}[Thm]{Lemma}

\theoremstyle{remark}
\newtheorem{Rem}[Thm]{Remark}

\theoremstyle{definition}

\renewcommand{\d}{{\rm d}}

\title[The Chevalley--Eilenberg complex and deformation quantization...]{The Chevalley--Eilenberg complex and deformation quantization in presence of two branes}
\author{C.~A.~Rossi}
\address{MPIM Bonn, Vivatsgasse 7, 53111 Bonn (Germany)}

\begin{document}



\begin{abstract}
In this note, we prove that, for a finite-dimensional Lie algebra $\mathfrak g$ over a field $\mathbb K$ of characteristic $0$ which contains $\mathbb C$, the Chevalley--Eilenberg complex $\mathrm U(\mathfrak g)\otimes \wedge(\mathfrak g)$, which is in a natural way a deformation quantization of the Koszul complex of $\mathrm S(\mathfrak g)$, is $A_\infty$-quasi-isomorphic to the deformation quantization of the $A_\infty$-bimodule $K=\mathbb K$ provided by the Formality Theorem in presence of two branes~\cite{CFFR}.
\end{abstract}

\maketitle

\section{Introduction}\label{s-intro}
Let $\mathbb K$ be a field of characteristic $0$. 
To a general finite-dimensional Lie algebra $\mathfrak g$ we may associate a natural complex, the Chevalley--Eilenberg chain complex $\mathrm{CE}(\mathfrak g)=\mathrm U(\mathfrak g)\otimes \wedge(\mathfrak g)$, where $\mathrm U(\mathfrak g)$, resp.\ $\wedge(\mathfrak g)$, denotes the UEA (short for Universal Enveloping Algebra), resp.\ exterior coalgebra, of $\mathfrak g$.
\begin{Rem}\label{r-alg-coalg}
In fact, $\wedge(\mathfrak g)$ is naturally a graded commutative algebra: in the framework of Koszul duality, $\wedge(\mathfrak g)$ has to be regarded as the Koszul dual coalgebra of $\mathrm S(\mathfrak g)$.
\end{Rem}
We observe further that we consider on $\mathrm{CE}(\mathfrak g)$ a non-standard non-positive grading, so as to make the differential of degree $1$.

It is obvious that $\mathrm{CE}(\mathfrak g)$ inherits a left $\mathrm U(\mathfrak g)$-action, which makes it into a complex of left $\mathrm U(\mathfrak g)$-modules; it is also well-known that the cohomology of $\mathrm{CE}(\mathfrak g)$ is concentrated in degree $0$, where it equals the augmentation module $\mathbb K$ over $\mathrm U(\mathfrak g)$.

Furthermore, $\mathrm{CE}(\mathfrak g)$ inherits {\em via} contraction a dg right action from the dg algebra $(\wedge(\mathfrak g^*),\mathrm d,\wedge)$ (the Chevalley--Eilenberg complex of $\mathfrak g$ with values in the trivial $\mathfrak g$-module $\mathbb K$), thus turning $\mathrm{CE}(\mathfrak g)$ into a dg $(\mathrm U(\mathfrak g),\cdot)$-$(\wedge(\mathfrak g^*),\mathrm d,\wedge)$-bimodule.
In particular, the complex $\mathrm{CE}(\mathfrak g)$ has a structure of $A_\infty$-$(\mathrm U(\mathfrak g),\cdot)$-$(\wedge(\mathfrak g^*),\mathrm d,\wedge)$-bimodule.

The complex $\mathrm{CE}(\mathfrak g)$ is actually the Koszul complex of the inhomogeneous Koszul algebra $\mathrm U(\mathfrak g)$, see {\em e.g.} the theory of inhomogeneous quadratic algebras developed in~\cite{Pr},~\cite[Section 3.6]{LV}. 

On the other hand, setting $A=\mathrm S(\mathfrak g)$ and $B=\wedge(\mathfrak g^*)$, both regarded as commutative dg algebras with trivial differential, the graphical techniques of Kontsevich permit us to endow $K=\mathbb K$ with a non-trivial structure of $A_\infty$-$A$-$B$-bimodule, see~\cite[Subsection 6.2]{CFFR} and later on for more details.
This $A_\infty$-bimodule structure is a key ingredient in the formulation of a Formality Theorem for the Poisson manifold $X=\mathfrak g^*$ in presence of the two coisotropic submanifolds $U_1=X$ and $U_2=\{0\}$ (observe that the natural Kirillov--Kostant Poisson structure on $X$ vanishes at $0$).
The (graded) Formality Theorem of Kontsevich~\cite{K} produces $i)$ an associative algebra $(\mathrm S(\mathfrak g),\star)$, with (non-commutative!) product $\star$ out of $A$ and $ii)$ a dg algebra $(\wedge(\mathfrak g^*),\mathrm d,\wedge)$ with standard wedge product and Chevalley--Eilenberg differential $\mathrm d$ out of $B$.
Moreover, the Formality Theorem in presence of two branes~\cite{CFFR} yields a corresponding deformation quantization of the $A_\infty$-$A$-$B$-bimodule $K$ into an $A_\infty$-$(\mathrm S(\mathfrak g),\star)$-$(\wedge(\mathfrak g^*),\mathrm d,\wedge)$-bimodule, which, by abuse of notation, is still denoted by $K$.

It has been proved {\em e.g.} in~\cite[Subsection 8.3]{K} or in~\cite[Subsection 3.2]{CFR} that $(\mathrm S(\mathfrak g),\star)$ is isomorphic to $(\mathrm U(\mathfrak g),\cdot)$ as an associative algebra; in~\cite{W}, it has been proved that the (left) augmentation module $\mathbb K$ over $\mathrm S(\mathfrak g)$ deforms to a left $(\mathrm S(\mathfrak g),\star)$-module.
Therefore, there is an isomorphism of dg bimodules from the dg $(\mathrm U(\mathfrak g),\cdot)$-$(\wedge(\mathfrak g^*),\mathrm d,\wedge)$-bimodule $\mathrm{CE}(\mathfrak g)$ to the $(\mathrm S(\mathfrak g),\star)$-$(\wedge(\mathfrak g^*),\mathrm d,\wedge)$-bimodule $\mathrm S(\mathfrak g)\otimes\wedge(\mathfrak g)$ with due changes in the differential and in the left module structure. 
\begin{Thm}\label{t-CE-bar}
The (deformed) $A_\infty$-$(\mathrm S(\mathfrak g),\star)$-$(\wedge(\mathfrak g^*),\mathrm d,\wedge)$-bimodule $K$ is $A_\infty$-quasi-isomorphic to the $A_\infty$-$(\mathrm S(\mathfrak g),\star)$-$(\wedge(\mathfrak g^*),\mathrm d,\wedge)$-bimodule $\mathrm S(\mathfrak g)\otimes\wedge(\mathfrak g)$.
\end{Thm}
For a finite-dimensional $\mathbb K$-vector space $V$, it has been proved in~\cite{FRW} that the Koszul complex $\mathrm K(A)$ of the Koszul algebra $A=\mathrm S(V)$, viewed as a dg $A$-$B$-bimodule, $B=\wedge(V^*)$, is $A_\infty$-quasi-isomorphic to $K=\mathbb K$ with the non-trivial $A_\infty$-$A$-$B$-bimodule structure described in~\cite{CFFR}.

The Chevalley--Eilenberg complex $\mathrm{CE}(\mathfrak g)$, as already remarked, is the Koszul complex of the  quadratic-linear Koszul algebra $\mathrm U(\mathfrak g)$: $\mathrm{CE}(\mathfrak g)$ admits a nice description in terms of deformation quantization of the Koszul complex of $A=\mathrm S(\mathfrak g)$.

The natural question arises, whether $\mathrm{CE}(\mathfrak g)$ is $A_\infty$-quasi-isomorphic to the deformation quantization of $K$ in the sense of~\cite{CFFR}: Theorem~\ref{t-CE-bar} provides a positive answer to this question. 

The strategy of the proof mimics the one adopted in the proof of~\cite[Theorem 1.2]{FRW}: namely, setting (again by abuse of notation) $A=(\mathrm S(\mathfrak g),\star)$ and $B=(\wedge(\mathfrak g^*),\mathrm d,\wedge)$, we prove that both morphisms in the sequence 
\[
S(\mathfrak g)\otimes\wedge(\mathfrak g)\hookrightarrow A\underline\otimes_A K \to K
\]
are quasi-isomorphisms of $A_\infty$-bimodules.
$(-\underline\otimes_A-)$ denotes the tensor product of $A_\infty$-bimodules, defined in Section~\ref{s-1}. 
Since both $A$ and $B$ are strictly unital algebras and $K$ is a strictly unital $A_\infty$-$A$-$B$-bimodule, the $A_\infty$-tensor product $A\underline\otimes_AK$ is $A_\infty$-quasi-isomorphic to $K$. 
On the other hand, the natural inclusion from $S(\mathfrak g)\otimes\wedge(\mathfrak g)\hookrightarrow A\underline{\otimes}_A K$ defines a strict $A_\infty$-quasi-isomorphism of $A_\infty$-bimodules: this is the non-trivial part of the statement, and is proved by inspecting carefully the deformed $A_\infty$-bimodule structure on $K$.

 \subsection*{Acknowledgments} We thank B.~Vallette for having carefully read a first version of this note and for having pointed out the connection to inhomogeneous quadratic algebras and related Koszul duality theory, M.~Duflo for having elucidated to us the idea which has been presented in Subsection~\ref{ss-2-2} and for having carefully read a first version of this note, and the anonymous referee for many useful comments and suggestions.
 
 \section{Notation and conventions}\label{s-0}
Throughout the paper, $\mathbb K$ is a field of characteristic $0$, which contains $\mathbb C$. 
Let $\mathfrak g$ be a finite-dimensional Lie algebra over $\mathbb K$ and $\mathfrak g^*$ its dual over $\mathbb K$.
Further, we denote by $\{x_i\}$, $i=1,\dots,d=\dim V$, a basis of $\mathfrak g$: this specifies automatically global linear coordinates on $\mathfrak g^*$.
We further denote by $\pi$ the Kirillov--Kostant--Souriau linear Poisson structure on $\mathfrak g^*$: if we consider the algebra $\mathcal O(X)$ of global regular functions on $X=\mathfrak g^*$, we have $\mathcal O(X)=\mathrm S(\mathfrak g)$ and the Lie bracket on $\mathfrak g$ extends to a biderivation $\{\bullet,\bullet\}$ on $\mathcal O(X)$, which obviously satisfies the Jacobi identity, hence defines a Poisson bracket on $\mathcal O(X)$.
The corresponding Poisson bivector field $\pi$ is expressed w.r.t.\ the coordinates $\{x_i\}$ {\em via} $\pi=f_{ij}^kx_k\partial_i\partial_j$, suppressing wedge products for the sake of brevity.

Let $\texttt{grMod}_\mathbb K$ be the monoidal category of graded vector spaces, with graded tensor product, and with inner spaces of morphisms (i.e.\ we consider morphisms, which are finite sums of morphisms of any degree); $[\bullet]$ denotes the degree-shifting functor on $\texttt{grMod}_\mathbb K$.
In particular, the identity morphism of an object $M$ of $\texttt{grMod}_\mathbb K$ induces a canonical isomorphism $s:M\to M[1]$ of degree $-1$ with inverse $s^{-1}:M[1]\to M$ (suspension and de-suspension isomorphisms): for the sake of simplicity, we will Cartan's notation
\[
(v_1|\cdots|v_n)=s(v_1)\otimes\cdots \otimes s(v_n).
\]
The degree of an element $m$ of a homogeneous component of an object $M$ of $\texttt{grMod}_\mathbb K$ is denoted by $|m|$.
Unadorned tensor products are meant to be over $\mathbb K$.

An $A_\infty$-algebra structure over $A$, an object of $\texttt{grMod}_\mathbb K$, is equivalent to the existence of a codifferential on the cofree coalgebra with counit on $\mathrm T(A[1])=\bigoplus_{n\geq 0}A[1]^{\otimes n}$ cogenerated by $A[1]$. 
The codifferential $\mathrm d_A$ is uniquely determined by its Taylor components
\[
\mathrm d_A^n:A[1]^{\otimes n}\to A[1],\ n\geq 0,
\]
all of degree $1$, and the condition that $\mathrm d_A$ squares to $0$ translates into an infinite family of quadratic relations between its Taylor components.
We further set $\mathrm m_A^n=(-1)^{\frac{n(n-1)}2}s^{-1}\circ \mathrm d_A^n \circ s^{\otimes n}$. 
By construction, $\mathrm m_A^n$ are $\mathbb K$-linear maps from $A^{\otimes n}$ to $A$ of degree $2-n$.
We refer to $\mathrm m_A^0$ as to the curvature of $A$: it is an element of $A$ of degree $2$, which measures the failure of $(A,\mathrm m_A^1)$ to be a dg (short for ``differential graded'') vector space over $\mathbb K$.
If $\mathrm m_A^0=0$, then $A$ is said to be flat.
\begin{Rem}\label{r-dga}
A dg algebra $(A,\mathrm d_A,\mathrm m_A)$ is a flat $A_\infty$-algebra by means of the Taylor components
\[
\mathrm d_A^1=s\circ \mathrm d_A\circ s^{-1},\ \mathrm d_A^2=-s\circ\mathrm m_A\circ (s^{-1})^{\otimes 2},\ \mathrm d_A^n=0,\ n\geq 3.
\]
In particular, a flat $A_\infty$-structure on a dg vector space $A$ concentrated in degree $0$ is equivalent to an associative algebra structure on $A$.
\end{Rem}
Given two $A_\infty$-algebras $A$, $B$, an $A_\infty$-$A$-$B$-bimodule structure on an object $K$ of $\texttt{grMod}_\mathbb K$ is equivalent to the existence of a codifferential on the cofree bicomodule $\mathrm T(A[1])\otimes K[1]\otimes \mathrm T(B[1])$ which is compatible with the codifferentials on $\mathrm T(A[1])$ and $\mathrm T(B[1])$. 
As for $A_\infty$-algebras, such a codifferential $\mathrm d_K$ is uniquely determined by its Taylor components
\[
\mathrm d_K^{m,n}:A[1]^{\otimes m}\otimes K[1]\otimes B[1]^{\otimes n}\to K[1],\ m,n\geq 0,
\]
all of degree $1$.
As before, we introduce the maps $\mathrm m_K^{m,n}=(-1)^{\frac{(m+n)(m+n+1)}2}s^{-1}\circ \mathrm d_K^{m,n}\circ s^{\otimes m+1+n}$, of degree $1-m-n$.
The condition that $\mathrm d_K$ squares to $0$ is equivalent to an infinite family of quadratic relations between the Taylor components of $\mathrm d_A$, $\mathrm d_B$ and $\mathrm d_K$.
For more details on $A_\infty$-bimodules over $A_\infty$-algebras, we refer to~\cite[Sections 3-4]{CFFR}.
\begin{Rem}\label{r-dgbimod}
Given two dg algebras $(A,\mathrm d_A,\mathrm m_A)$ and $(B,\mathrm d_B,\mathrm m_B)$, which in virtue of Remark~\ref{r-dga} may be regarded as flat $A_\infty$-algebras, a dg $A$-$B$-bimodule structure on $K$ is equivalent to an $A_\infty$-$A$-$B$-bimodule structure on $K$ with Taylor components
\[
\mathrm d_{K}^{0,0}=s\circ \mathrm d_K\circ s^{-1},\ \mathrm d_K^{1,0}=-s\circ \mathrm m_L\circ (s^{-1})^{\otimes 2},\ \mathrm d_K^{0,1}=-s\circ \mathrm m_R\circ (s^{-1})^{\otimes 2},\ \mathrm d_K^{m,n}=0,\ m+n\geq 2, 
\]
where $\mathrm d_K$, resp.\ $\mathrm m_L$, resp.\ $\mathrm m_R$, denotes the differential, resp.\ the left $A$-, resp.\ the right $B$-action, on $K$.
\end{Rem}
It is not difficult to verify that an $A_\infty$-algebra $A$ can be turned easily into an $A_\infty$-$A$-$A$-bimodule by declaring $\mathrm d_A^{m,n}=\mathrm d_A^{m+n+1}$, $m,n\geq 0$: this obvious observation will play an important role in later computations.

It is important to observe that, if $A$ and $B$ are both flat, then an $A_\infty$-$A$-$B$-bimodule structure on $K$ restricts to a left $A_\infty$-$A$- and right $A_\infty$-$B$-module structure on $K$ respectively in the sense of~\cites{Kel,L-H}.
On the other hand, if either $A$ or $B$ or both have non-trivial curvature, the $A_\infty$-bimodule structure does not restrict to (left or right) $A_\infty$-module structures, see e.g.~\cite{W} and~\cite[Subsection 4.1]{CFFR}.

An $A_\infty$-algebra $A$ is said to be {\bf strictly unital}, if it possesses an element $1$ of degree $0$, such that
\[
\mathrm m_A^2(1\otimes a)=\mathrm m_A^2(a\otimes 1)=a,\quad \mathrm m_A^n(a_1\otimes \cdots\otimes a_n)=0,\ n\neq 2,
\]
if $a_i=1$, for some $i=1,\dots,n$.
If $A$ is strictly unital, and $K$ is an $A_\infty$-$A$-$B$-bimodule, then $K$ is strictly (left-)unital w.r.t.\ $A$, if the identities hold true
\[
\mathrm m_K^{1,0}(1\otimes k)=k,\quad \mathrm m_K^{m,n}(a_1\otimes \cdots\otimes a_m\otimes k\otimes b_1\otimes \cdots\otimes b_n)=0,\ m\neq 1,\ n\geq 0,
\]
if $a_i=1$, for some $i=1,\dots,m$.
Similarly, one defines a strictly (right-)unital $A_\infty$-bimodule structure on $K$. 

Given two $A_\infty$-algebras $A$, $B$, an $A_\infty$-morphism from $A$ to $B$ is a coalgebra morphism $\varphi:\mathrm T(A[1])\to \mathrm T(B[1])$ of degree $0$ and compatible with the respective codifferentials.
The cofreeness of $\mathrm T(A[1])$ and $\mathrm T(B[1])$ implies that an $A_\infty$-morphism $\varphi$ is uniquely determined by its Taylor components 
\[
\varphi^n:A[1]^{\otimes n}\to B[1],\ n\geq 0,\ \text{all of degree $0$}.
\]
Similarly, given two $A_\infty$-algebras $A$, $B$ and two $A_\infty$-$A$-$B$-bimodules $K_1$, $K_2$, an $A_\infty$-morphism from $K_1$ to $K_2$ is a morphism of bicomodules $\psi:\mathrm T(A[1])\otimes K_1[1]\otimes\mathrm T(B[1])\to\mathrm T(A[1])\otimes K_1[1]\otimes\mathrm T(B[1])$ of degree $0$ and compatible with the respective codifferentials: as for $A_\infty$-morphisms between $A_\infty$-algebras, $\psi$ is uniquely determined by its Taylor components 
\[
\psi^{m,n}:A[1]^{\otimes m}\otimes K_1[1]\otimes B[1]^{\otimes n}\to K_2[1],\ m,n\geq 0.
\]
Of course, the compatibility with the codifferentials translates into a complicated infinite family of polynomial identities w.r.t.\ the Taylor components of all codifferentials and morphisms (which are linear w.r.t.\ the Taylor components of the codifferentials).

A morphism $\varphi$ of $A_\infty$-algebras, resp.\ $\psi$ of $A_\infty$-bimodules, is said to be {\bf strict}, if its only non-trivial Taylor component is $\varphi^1$, resp.\ $\psi^{0,0}$. 

Finally, we denote by $\hbar$ a formal parameter (``Planck's constant'').
For an object $V$ of $\texttt{grMod}_\mathbb K$, we set $V_{\hbar}=V[\![\hbar]\!]$: it is a graded, topologically free $\mathbb K[\![\hbar]\!]$-module (here, the degree of $\hbar$ is set to be $0$).
In particular, we may consider the category $\texttt{grMod}_{\mathbb K[\![\hbar]\!]}$ of graded, topologically free $\mathbb K[\![\hbar]\!]$-modules: it is a symmetric monoidal category with the topological tensor product over $\mathbb K[\![\hbar]\!]$.

\section{The $A_\infty$-bar construction}\label{s-1}
In this Section, we briefly review the main features of the tensor product between $A_\infty$-bimodules, focusing on the $A_\infty$-bar construction associated to a strictly unital $A_\infty$-$A$-$B$-bimodule $K$ over two flat, strictly unital $A_\infty$-algebras $A$, $B$: we refer to~\cite[Section 3]{FRW} for more details, recalling here only the main formul\ae\ needed for later computations.

For three $A_\infty$-algebras $A$, $B$ and $C$, an $A_\infty$-$A$-$B$-bimodule $K_1$ and an $A_\infty$-$B$-$C$-bimodule $K_2$, we consider the tensor product of $K_1$ and $K_2$ over $B$, as an element of $\texttt{grMod}_\mathbb K$
\[
K_1\underline\otimes_B K_2=K_1\otimes \mathrm T(B[1])\otimes K_2,
\]
whose $A_\infty$-bimodule structure can be explicitly given in terms of its Taylor components {\em via}
\begin{equation}\label{eq-tayl-tens}
\begin{aligned}
&\mathrm d_{K_1\underline\otimes_B K_2}^{m,n}\!(a_1|\cdots|a_m|k_1\otimes (b_1|\cdots|b_q)\otimes k_2|c_1|\cdots|c_n)=0,\quad m,n>0\\
&\mathrm d_{K_1\underline\otimes_B K_2}^{m,0}\!(a_1|\cdots|a_m|k_1\otimes (b_1|\cdots|b_q)\otimes k_2)=\sum_{l=0}^q s\left(s^{-1}(\mathrm d_{K_1}^{m,l}(a_1|\cdots|a_m|k_1|b_1|\cdots|b_l))\otimes (b_{l+1}|\cdots|b_q)\otimes k_2\right),\quad m>0\\
&\mathrm d_{K_1\underline\otimes_B K_2}^{0,n}\!(k_1\otimes (b_1|\cdots|b_q)\otimes k_2|c_1|\cdots|c_n)=(-1)^{|k_1|+\sum_{j=1}^q(|b_j|-1)}\sum_{l=0}^q s\left(k_1\otimes (b_1|\cdots|b_l)\otimes\right.\\
&\phantom{\mathrm d_{K_1\underline\otimes_B K_2}^{0,n}\!(k_1\otimes (b_1|\cdots|b_q)\otimes k_2|c_1|\cdots|c_n)=}\left.s^{-1}(\mathrm d_{K_1}^{q-l,n}(b_{l+1}|\cdots|b_q|k_2|c_1|\cdots|c_n)\right),\quad n>0,\\
&\mathrm d_{K_1\underline\otimes_B K_2}^{0,0}\!\left(s(k_1\otimes (b_1|\cdots|b_q)\otimes k_2)\right)=\sum_{l=0}^q s\left(s^{-1}(\mathrm d_{K_2}^{0,l}(k_1|b_1|\cdots|b_l)\otimes (b_{l+1}|\cdots|b_q)\otimes k_2\right)+\\
&\phantom{\mathrm d_{K_1\underline\otimes_B K_2}^{0,0}\!\left(s(k_1\otimes (b_1|\cdots|b_q)\otimes k_2)\right)=}+\sum_{0\leq l\leq q\atop 0\leq p\leq q-l} (-1)^{(|k_1|-1)+\sum_{j=1}^l (|b_j|-1)}s\!\left(k_1\otimes (b_1|\cdots|\mathrm d_B^p(b_{l+1}|\cdots|b_{l+p})|\cdots|b_q)\otimes k_2\right)+\\
&\phantom{\mathrm d_{K_1\underline\otimes_B K_2}^{0,0}\!\left(s(k_1\otimes (b_1|\cdots|b_q)\otimes k_2)\right)=}+(-1)^{|k_1|+\sum_{j=1}^q(|b_j|-1)}\sum_{l=0}^q
s\!\left(k_1\otimes (b_1|\cdots|b_l)\otimes s^{-1}(\mathrm d_{K_2}^{q-l,0}(b_{l+1}|\cdots|b_q|k_2)\right).
\end{aligned}
\end{equation}
That the Taylor components~\eqref{eq-tayl-tens} truly describe an $A_\infty$-$A$-$C$-bimodule structure on $K_1\underline{\otimes}_B K_2$ can be checked by straightforward computations; a more conceptual proof has been given in~\cite[Proposition 3.3]{FRW}.
We observe that, if both $A$, $C$ are flat, the Taylor component $m_{K_1\underline\otimes_B K_2}^{0,0}$ yields a structure of dg vector space on $K_1\underline\otimes_B K_2$ (while $B$ may have non-trivial curvature).

Now, let $A$, $B$ be two $A_\infty$-algebras, and $K$ be an $A_\infty$-$A$-$B$-bimodule $K$.
We may thus form the $A_\infty$-$A$-$B$-bimodule $A\underline\otimes_A K$ (viewing $A$ as an $A_\infty$-$A$-$A$-bimodule); similarly, we may consider the $A_\infty$-$A$-$B$-bimodule $K\underline\otimes_B B$.

There is a natural $A_\infty$-morphism $\mu$ of $A_\infty$-$A$-$B$-bimodules from $A\underline\otimes_A K$ to $K$, whose Taylor components are given explicitly by
\begin{equation}\label{eq-bar-constr}
\begin{aligned}
&\mu^{m,n}(a_1|\cdots|a_m|a\otimes (\widetilde a_1|\cdots|\widetilde a_q)\otimes k|b_1|\cdots|b_n)=\\
&=(-1)^{\sum_{i=1}^m (|a_i|-1)+|a|+\sum_{j=1}^q(|\widetilde a_j|-1)}\mathrm d_K^{m+1+q,n}(a_1|\cdots|a_m|a|\widetilde a_1|\cdots|\widetilde a_q|k|b_1|\cdots|b_n),\quad m,n,q\geq 0.
\end{aligned}
\end{equation}
Similar formul\ae\ hold true for the case of the $A_\infty$-$A$-$B$-bimodule $K\underline\otimes_B B$.
\begin{Prop}\label{p-bar}
For two $A_\infty$-algebras $A$, $B$ and an $A_\infty$-$A$-$B$-bimodule $K$, there is a natural $A_\infty$-morphism $\mu$, defined by~\eqref{eq-bar-constr}, of $A_\infty$-$A$-$B$-bimodules from $A\underline\otimes_A K$ to $K$.

If $A$, $B$ are both flat, and $A$, $K$ are strictly (left-)unital, then the $A_\infty$-morphism~\eqref{eq-bar-constr} is an $A_\infty$-quasi-isomorphism.
\end{Prop}
We refer to~\cite[Subsection 3.1]{FRW} for a proof of Proposition~\ref{p-bar}.



\section{The Chevalley--Eilenberg complex of a finite-dimensional Lie algebra $\mathfrak g$}\label{s-2}
Let $\mathfrak g$ be as in Section~\ref{s-0}.
Its Chevalley--Eilenberg chain complex $\mathrm{CE}(\mathfrak g)$ is defined as 
\[
\mathrm{CE}^q(\mathfrak g)=\begin{cases}
\mathrm U(\mathfrak g)\otimes\wedge^{-q}(\mathfrak g),& q\leq 0,\\
\{0\},& q>0.
\end{cases}
\]
Observe that we use a non-positive grading on $\mathrm{CE}(\mathfrak g)$, making it actually into a cochain complex (this is actually different from the classical convention, but we prefer to deal with a differential of degree $1$); still, $\mathrm{CE}(\mathfrak g)$ is referred to as the Chevalley--Eilenberg chain complex in the main literature. 

We now make use of the identifications $\wedge(\mathfrak g)\cong \mathbb K[\theta_i]$ and $\wedge(\mathfrak g^*)\cong \mathbb K[\partial_{\theta_i}]$, $i=1,\dots,d$, $d=\dim\mathfrak g$, where $\theta_i$ is an odd variable of degree $-1$, {\em i.e.} $\theta_i\theta_j=-\theta_j\theta_i$, for $i$, $j$ in $\{1,\dots,d\}$.
Consequently, the derivative $\partial_{\theta_i}$ has degree $1$, acts on generators of $\mathbb K[\theta_i]$ {\em via} $\partial_{\theta_i}\theta_j=\delta_{ij}$ and further satisfies the graded Leibniz rule (from the left)
\[
\partial_{\theta_i}(f_1 f_2)=(\partial_{\theta_i}f_1)f_2+(-1)^{|f_1|}f_1(\partial_{\theta_i}f_2),\ f_j\in \mathbb K[\theta_i],\ j=1,2. 
\] 
In particular, we have the identification $\mathrm{CE}(\mathfrak g)=\mathrm U(\mathfrak g)[\theta_i]$ of graded vector spaces.

The Chevalley--Eilenberg differential $\mathrm d$ on $\mathrm{CE}(\mathfrak g)$ is given by
\begin{equation}\label{eq-CE-d}
\mathrm d_\mathrm{CE}(u(\theta_{i_1}\cdots \theta_{i_q}))=\sum_{k=1}^q(-1)^{k+1}(u\cdot x_{i_k})(\theta_{i_1}\cdots \widehat{\theta_{i_k}}\cdots \theta_{i_q})+\sum_{1\leq k<l\leq q}(-1)^{k+l}f_{i_k,i_l}^j u(\theta_{j}\theta_{i_1}\cdots\widehat{\theta_{i_k}}\cdots\widehat{\theta_{i_l}}\cdots \theta_{i_q}),  
\end{equation}
where $\cdot$ denotes the associative product on $\mathrm U(\mathfrak g)$, and $f_{ij}^k$ denote the structure constants of $\mathfrak g$ w.r.t.\ the chosen basis.

The associative algebra structure on $\mathrm U(\mathfrak g)$ makes $\mathrm{CE}(\mathfrak g)$ into a left $\mathrm U(\mathfrak g)$-module in an obvious way: we denote the left action by $\mathrm m_L$, thus
\begin{equation}\label{eq-CE-l}
\mathrm U(\mathfrak g)\otimes \mathrm{CE}(\mathfrak g)\ni u_1\otimes (u\theta^I)\overset{\mathrm m_L}\mapsto (u_1\cdot u)\theta^I\in\mathrm{CE}(\mathfrak g),
\end{equation}
for some ordered multi-index $I$.

Now, $\mathbb K[\partial_{\theta_i}]$ acts naturally on $\mathrm{CE}(\mathfrak g)$ from the left: therefore, we may turn $\mathrm{CE}(\mathfrak g)$ into a right $\wedge(\mathfrak g^*)$-module simply {\em via}
\begin{equation}\label{eq-CE-r}
\mathrm{CE}(\mathfrak g)\otimes \wedge(\mathfrak g^*)\cong \mathrm U(\mathfrak g)[\theta_i]\otimes\mathbb K[\partial_{\theta_i}]\ni (u\theta^I)\otimes \partial_\theta^J\overset{\mathrm m_R}\mapsto (-1)^{|I||J|} u(\partial_\theta^J(\theta^I))\in \mathrm U(\mathfrak g)[\theta_i]\cong \mathrm{CE}(\mathfrak g), 
\end{equation}
for two ordered multi-indices $I$, $J$.

Observe that, since $\mathfrak g$ is a Lie algebra, $\wedge(\mathfrak g^*)\cong\mathbb K[\partial_{\theta_i}]$ is endowed with a map of degree $1$, which we denote by $\mathrm d$, which is uniquely defined on the generators $\partial_{\theta_i}$ by the formula $\mathrm d(\partial_{\theta_i})=f_{jk}^i\partial_{\theta_j}\partial_{\theta_k}$ extended to the whole algebra by means of the graded Leibniz rule (from the left): the Jacobi identity for the Lie bracket on $\mathfrak g$ implies that $\mathrm d$ squares to $0$.

We recall the augmentation map $\varepsilon$ from $\mathrm U(\mathfrak g)$ to $\mathbb K$.
\begin{Prop}\label{p-CE-bimod}
The morphisms~\eqref{eq-CE-d},~\eqref{eq-CE-l} and~\eqref{eq-CE-r} endow $\mathrm{CE}(\mathfrak g)$ with a dg $(\mathrm U(\mathfrak g),\cdot)$-$(\wedge(\mathfrak g^*),\mathrm d,\wedge)$-bimodule structure; further, the augmentation map $\varepsilon$ defines in an obvious way a quasi-isomorphism of dg bimodules from $\mathrm{CE}(\mathfrak g)$ to $\mathbb K$, where the latter is endowed with the obvious structure of $(\mathrm U(\mathfrak g),\cdot)$-$(\wedge(\mathfrak g^*),\mathrm d,\wedge)$-bimodule.
\end{Prop}
\begin{proof}
The fact that~\eqref{eq-CE-d} squares to $0$ is an elementary check using the Jacobi identity for $\mathfrak g$ and the fact that the commutator of two elements of $\mathfrak g$ inside $\mathrm U(\mathfrak g)$ equals their Lie bracket (as an element of $\mathfrak g$); since $\mathrm U(\mathfrak g)$ is concentrated in degree $0$, one has only to verify that the left action~\eqref{eq-CE-l} is $\mathrm U(\mathfrak g)$-linear, which is clear from its definition.

It remains to prove the following identity
\[
\mathrm d_\mathrm{CE}(\mathrm m_R((u\theta^I)\otimes \partial_\theta^J)=\mathrm m_R(\mathrm d_\mathrm{CE}(u\theta^I)\otimes \partial_\theta^J)+(-1)^{|I|}\mathrm m_R((u\theta^I)\otimes \mathrm d(\partial_\theta^J)),\ u\theta^I\in \mathrm{CE}(\mathfrak g),\ \partial_\theta^J\in\wedge(\mathfrak g^*),
\]
where we have used the identifications $\mathrm{CE}(\mathfrak g)\cong \mathrm U(\mathfrak g)[\theta_i]$ and $\wedge(\mathfrak g^*)\cong\mathbb K[\partial_{\theta_i}]$.
The proof of the previous identity is equivalent to the proof that the left action of $\mathbb K[\partial_{\theta_i}]$ on $\mathrm U(\mathfrak g)[\theta_i]$ by $\mathrm U(\mathfrak g)$-linear translation-invariant differential operators is compatible with the corresponding differentials.

The $\mathrm U(\mathfrak g)$-linearity of the translation invariant differential operators implies that it suffices to show 
\begin{equation}\label{eq-left-diff}
\mathrm d_\mathrm{CE}(\partial_\theta^I(\theta^J))=(\mathrm d(\partial_\theta^I))(\theta^J)+(-1)^{|J|}\partial_\theta^I(\mathrm d_\mathrm{CE}\theta^J),
\end{equation}
for two multi-indices $I=(i_1,\dots,i_p)$ and $J=(j_1,\dots,j_q)$, such that $1\leq i_1<\cdots<i_p\leq d$, $1\leq j_1<\cdots<j_q\leq d$ and $\partial_\theta^I=\partial_{\theta_{i_1}}\cdots\partial_{\theta_{i_p}}$, $\theta^J=\theta_{j_1}\cdots\theta_{j_q}$.
Identity~\eqref{eq-left-diff} is trivially satisfied if $p=|I|>|J|+1=q+1$: thus, we assume $p\leq q+1$.

Explicitly, we get
\[
\mathrm d(\partial_\theta^I)=\sum_{I_1\subseteq I\atop |I_1|=1}\varepsilon(I_1,I)f_{jk}^{I_1}\partial_{\theta}^{\{j,k\}\sqcup I\smallsetminus I_1},\ \mathrm d_\mathrm{CE}(\theta^J)=\sum_{J_1\subseteq J\atop |J_1|=1}\varepsilon(J_1,J)x_{J_1}\theta^{I\smallsetminus I_1}+\sum_{J_1,J_2\subseteq J\atop |J_1|=|J_2|=1,\ J_1<J_2}\varepsilon(J_1,J_2,J)f_{J_1,J_2}^k\theta^{\{k\}\sqcup I\smallsetminus I_1},
\]
where {\em e.g.} the sign $\varepsilon(I_1,I)$, for an ordered subset $I_1$ of $I$, is uniquely specified by $\theta^I=\varepsilon(I_1,I)\theta^{I_1}\theta^{I\smallsetminus I_1}$. 
Identity~\eqref{eq-left-diff} follows then from the previous two expressions by a direct computation.

The left $(\mathrm U(\mathfrak g),\cdot)$-action and the right $(\wedge(\mathfrak g^*),\mathrm d,\wedge)$-action on $\mathrm{CE}(\mathfrak g)$ are compatible in virtue of the aforementioned $\mathrm U(\mathfrak g)$-linearity of the right $\wedge(\mathfrak g^*)$-action.

See {\em e.g.}~\cite[Theorem 7.7.2]{Weib} for a proof of the fact that $\varepsilon$ is a quasi-isomorphism.
\end{proof}
In particular, Proposition~\ref{p-CE-bimod} implies that $\mathrm{CE}(\mathfrak g)$ inherits a structure of $A_\infty$-$(\mathrm U(\mathfrak g),\cdot)$-$(\wedge(\mathfrak g^*),\mathrm d,\wedge)$-bimodule, whose cohomology is concentrated in degree $0$ and equals $\mathbb K$.

\subsection{The Chevalley--Eilenberg complex as a Koszul complex}\label{ss-2-1}
We highlight in the present subsection the fact that $\mathrm{CE}(\mathfrak g)$, for $\mathfrak g$ as above, is the Koszul complex of $\mathrm U(\mathfrak g)$.

First of all, $\mathrm U(\mathfrak g)$ is a quadratic-linear algebra using the language of~\cite[Section 3.6]{LV}, {\em i.e.} it is the quotient of a free algebra w.r.t.\ the two-sided ideal generated by $R$, a linear subspace of $\mathfrak g\oplus \mathfrak g^{\otimes 2}$.

The construction of Koszul resolutions for quadratic-linear algebras traces back to~\cite{Pr}, who regards the Chevalley--Eilenberg complex $\mathrm{CE}(\mathfrak g)$ as the Koszul complex of $\mathrm U(\mathfrak g)$; the general theory of Koszul duality for quadratic-linear algebras may be found in~\cite[Section 3.6]{LV}.
We recall in the following some relevant features thereof without entering into the details.

In the case at hand, the two-sided ideal $R$ is spanned by elements of the form $x_i\otimes x_j-x_j\otimes x_i-[x_i,x_j]$, for a basis $\{x_i\}$ of $\mathfrak g$ as above.
Observe that the linear term can be regarded as the image of $x_i\otimes x_j-x_j\otimes x_i$ in $\mathfrak g^{\otimes 2}$ w.r.t.\ the Lie bracket onto $\mathfrak g$.

The Koszul dual dg coalgebra of $\mathrm U(\mathfrak g)$ identifies as a graded vector space with $\wedge(\mathfrak g)$: a more proper notation for the Koszul dual coalgebra would be $\wedge^c(\mathfrak g)$ to highlight the fact that $\wedge(\mathfrak g)$ is endowed with the natural shuffle coproduct.
The codifferential $\mathrm d_2$ on $\wedge^c(\mathfrak g)$ is induced by the Lie bracket, which determines the linear term in quadratic-linear relations for $\mathrm U(\mathfrak g)$: it is given by the second term on the right-hand side of~\eqref{eq-CE-d}, setting $u=1$.
We observe that in the standard theory of Koszul resolutions, $\wedge^c(\mathfrak g)$ would be positively graded and the codifferential would have degree $-1$: we prefer to use the non-standard non-positive grading to make the codifferential of degree $1$.

The Koszul dual dg algebra of $\mathrm U(\mathfrak g)$ is the graded dual algebra of $\left(\wedge^c(\mathfrak g),\Delta,\mathrm d_2\right)$: it readily identifies with the cohomological Chevalley--Eilenberg complex $\left(\wedge(\mathfrak g^*),\wedge,\mathrm d\right)$.
The grading on $\left(\wedge(\mathfrak g^*),\wedge,\mathrm d\right)$ is non-negative by our non-standard choice of grading on the Koszul dual dg coalgebra.

The first term $\mathrm d_1$ on the right-hand side of~\eqref{eq-CE-d} admits an interpretation in the theory of Koszul duality for quadratic-linear algebras: the identity map of $\mathfrak g$ induces an isomorphism of degree $-1$ from $\mathfrak g$ to $\mathfrak g[1]$, hence a morphism $\kappa$ from $\wedge(\mathfrak g)$ to $\mathrm U(\mathfrak g)$, where $\wedge(\mathfrak g)=\wedge^c(\mathfrak g)$ but with non-positive grading {\em via}
\[
\wedge(\mathfrak g)=\mathrm S(\mathfrak g[1])\twoheadrightarrow \mathfrak g[1]\cong \mathfrak g\hookrightarrow \mathrm U(\mathfrak g),
\]
where the first map is simply the natural projection onto the piece of degree $-1$.
The map $\kappa$ is a twisting cocycle as defined in~\cite[Chapter 2]{LV}, and $\mathrm d_1=\mathrm d_\kappa$, the twisted differential induced by $\kappa$.

We finally observe that $\wedge(\mathfrak g)$ is a dg subcoalgebra of the bar complex of $\mathrm U(\mathfrak g)$: the natural inclusion from $\wedge(\mathfrak g)\hookrightarrow \mathrm T(\mathrm U(\mathfrak g))$ will be revisited in Section~\ref{s-5}, Formula~\eqref{eq-skew}.

\subsection{The Chevalley--Eilenberg complex as a deformation quantization of the Koszul complex of $\mathrm S(\mathfrak g)$}\label{ss-2-2}
We discuss here briefly an approach to the Chevalley--Eilenberg complex $\mathrm{CE}(\mathfrak g)$ in the framework of deformation quantization: its has been suggested to us by M.~Duflo, but whose origin traces back to P.~Cartier according to M.~Duflo.

For a Lie algebra $\mathfrak g$ over $\mathbb K$ as above, we consider the graded vector space $\widehat{\mathfrak g}=\mathfrak g\oplus\mathfrak g[1]$ concentrated in degrees $-1$, $0$.
The adjoint representation of $\mathfrak g$ on itself endows $\widehat{\mathfrak g}$ with the structure of a graded Lie algebra.

Observe that the symmetric algebra $\mathrm S(\widehat{\mathfrak g})$ identifies canonically with the graded vector space $\mathrm S(\mathfrak g)\otimes\wedge(\mathfrak g)$ underlying the Koszul complex of $A=\mathrm S(\mathfrak g)$.

Since $\widehat{\mathfrak g}$ is a finite-dimensional graded Lie algebra, we may consider its UEA $\mathrm U(\widehat{\mathfrak g})$, which has an obvious structure of graded associative algebra: there is an isomorphism of graded vector spaces from $\mathrm S(\widehat{\mathfrak g})$ to $\mathrm U(\widehat{\mathfrak g})$, the PBW isomorphism.
If $\{x_i\}$ denotes a basis of $\mathfrak g$ and $\{\theta_i\}$ a basis of $\mathfrak g[1]$, we may consider $\{x_1,\dots,x_d,\theta_1,\dots,\theta_d\}$ as an ordered basis of $\widehat{\mathfrak g}$: this yields a PBW basis of $\mathrm U(\widehat{\mathfrak g})$, which in turn permits to identify $\mathrm U(\widehat{\mathfrak g})$ with $\mathrm U(\mathfrak g)\otimes\wedge(\mathfrak g)$. 

Obviously, $\mathfrak g$ is a Lie subalgebra of $\widehat{\mathfrak g}$, hence $\mathrm U(\mathfrak g)$ is a subalgebra of $\mathrm U(\widehat{\mathfrak g})$: the previous identification implies that $\mathrm U(\mathfrak g)\otimes\wedge(\mathfrak g)$ is a left $\mathrm U(\mathfrak g)$-module with the obvious module structure.

The Koszul differential on $\mathrm S(\widehat{\mathfrak g})$ can be twisted w.r.t.\ the PBW isomorphism: upon the identification $\mathrm U(\widehat{\mathfrak g})=\mathrm U(\mathfrak g)\otimes\wedge(\mathfrak g)$, the twisted Koszul differential $\mathrm d_\mathrm K$ equals the differential~\eqref{eq-CE-d}.
Namely, the twisted Koszul differential is a graded derivation of $\mathrm U(\mathfrak g)\otimes\wedge(\mathfrak g)$ of degree $1$, and it suffices to evaluate it on the generators $\{x_i\}$, $\{\theta_i\}$, whence
\[
\begin{aligned}
\mathrm d_\mathrm K(u(\theta_{i_1}\cdots\theta_{i_q}))&=\sum_{k=1}^q(-1)^{k-1}u(\theta_{i_1}\cdots \theta_{i_{k-1}}x_{i_k}\theta_{i_{k+1}}\cdots\theta_{i_q})=\\
&=\sum_{k=1}^q(-1)^{k-1}(u\cdot x_{i_k})(\theta_{i_1}\cdots \widehat{\theta_{i_k}}\cdots\theta_{i_q})+\sum_{1\leq k<l\leq q}(-1)^{k+l}f_{i_k,i_l}^j u(\theta_{j}\theta_{i_1}\cdots\widehat{\theta_{i_k}}\cdots\widehat{\theta_{i_l}}\cdots \theta_{i_q}),
\end{aligned}
\]
where we have used the graded commutation relations $[x_i,\theta_j]=f_{ij}^k\theta_k$ and $[\theta_i,\theta_j]=0$ in $\mathrm U(\widehat{\mathfrak g})=\mathrm U(\mathfrak g)\otimes\wedge(\mathfrak g)$.
Observe that the Koszul complex $\mathrm S(\widehat{\mathfrak g})$ is acyclic by means of an explicit homotopy: twisting this homotopy w.r.t.\ the PBW isomorphism yields an explicit homotopy for the Chevalley--Eilenberg complex of $\mathfrak g$.
(Of course, to write down the explicit homotopy on $\mathrm U(\mathfrak g)\otimes\wedge(\mathfrak g)$ requires some work.)

The right $\wedge(\mathfrak g^*)$-action on $\mathrm S(\widehat{\mathfrak g})$ defines in an obvious way a right $wedge(\mathfrak g^*)$-action on $\mathrm U(\widehat{\mathfrak g})$: the fact that the former action does not intertwine $\mathfrak g$ with $\mathfrak g[1]$ implies that the latter action identifies with~\eqref{eq-CE-r}.

Further, $\mathrm U(\widehat{\mathfrak g})$ is a (graded) cocommutative coalgebra: $\mathrm U(\mathfrak g)$ and $\wedge(\mathfrak g)$ are obvious cocommutative subcoalgebras.

Thus, the previous arguments show a simple way to obtain the Chevalley--Eilenberg complex of $\mathfrak g$ from the Koszul complex of the graded Lie algebra $\widehat{\mathfrak g}$ by means of deformation quantization: namely, deformation quantization {\em \`a la} Kontsevich works in the graded case as well

\subsection{The dg bimodule structure of the Koszul complex}\label{ss-2-3}
We want finally to briefly discuss a nice, general feature of the Koszul complex $\mathrm K(A)$ of a quadratic algebra $A$ generated by a finite-dimensional vector space $V$ over $\mathbb K$.
\begin{Prop}\label{p-kosz-bim}
Let $A=\mathrm T(V)/R$ be a quadratic algebra generated by a finite-dimensional vector space $V$ over $\mathbb K$; $R$ denotes the two-sided ideal generated by a subspace $R$ of $V^{\otimes 2}$.
Then, the Koszul complex $\mathrm K(A)$ of $A$ has a structure of dg $A$-$A^!$-bimodule, where $A^!$ is the Koszul dual algebra of $A$.
\end{Prop}
\begin{proof}
We give a sketch of the proof.

If $A=\mathrm T(V)/R$, the Koszul complex $\mathrm K(A)$ can be written as $\mathrm K(A)=A\otimes A^{\text{!`}}$, where $B$ is the Koszul dual coalgebra of $A$.
More explicitly, $A^{\text{!`}}$ is a subcoalgebra of the cofree coassociative tensor coalgebra $\mathrm T^c(V)$ with counit cogenerated by $V$ given by
\[
(A^{\text{!`}})^0=\mathbb K,\ (A^{\text{!`}})^1=V,\ \mathrm T^{n,c}(V)\supseteq (A^{\text{!`}})^n=\bigcap_{i=0}^{n-2}V^{\otimes i}\otimes R\otimes V^{\otimes(n-2-i)},\ n\geq 2.
\]
We write $\Delta$ for the coproduct on $B$: the opposite coproduct $\Delta^\mathrm{op}$ is simply the composition of $\Delta$ with the natural twist on $A^{\text{!`}}\otimes A^{\text{!`}}$.
We denote by $\Delta^+$ the coproduct the cofree coassociative tensor coalgebra $\mathrm T^{c,+}(V)=\bigoplus_{n\geq 1} V^{\otimes n}$ without counit.

The Koszul dual algebra $A^!$ of $A$ is defined as the dual of $(A^{\text{!`}},\Delta^\mathrm{op})$: as a graded vector space, it identifies with $\mathrm T(V^*)/R^\perp$, where $R^\perp$ denotes the two-sided ideal of $\mathrm T(V^*)$ generated by the annihilator of $R$ in $V^*\otimes V^*$.
We observe that we use a different convention for the product on $A^!$: in fact, we use the opposite of the natural product in order to get a right $A^!$-action on $\mathrm K(A)$.

We define a map from $\mathrm K(A)\otimes A^!$ to $A^!$ {\em via} the composite
\[
\xymatrix{\mathrm K(A)\otimes A^!=A\otimes A^{\text{!`}}\otimes A^!\ar[rr]^-{1\otimes \Delta\otimes 1} & & A\otimes A^{\text{!`}}\otimes A^{\text{!`}}\otimes A^!\ar[rr]^-{1\otimes 1\otimes \mathrm{ev}} & & A\otimes A^{\text{!`}}=\mathrm K(A)},
\]
where $\mathrm{ev}$ denotes the duality pairing between $A^{\text{!`}}$ and $A^!$.
By the very definition of the multiplication in $A^!$ follows that the previous composite map defines a right $A^!$-action on $\mathrm K(A)$, which is obviously compatible with the left $A$-action.

It remains to prove that the above right $A^!$-action is compatible with the Koszul differential.
To see this, it is better to re-write the Koszul differential on $\mathrm K^{-n}(A)$, $n\geq 2$, as  
\[
\xymatrix{\mathrm K^{-n}(A)=A\otimes (A^{\text{!`}})^{-n}\ar[r]^-{1\otimes \Delta^+} & \bigoplus_{p+q=n\atop p,q\geq 1}A\otimes (A^{\text{!`}})^{-p}\otimes (A^{\text{!`}})^{-q}\ar@{->>}[r] & A\otimes V \otimes (A^{\text{!`}})^{-n+1}\ar[r]^-{\mu_A\otimes 1} & A\otimes (A^{\text{!`}})^{-n+1}=\mathrm K^{-n+1}(A),\ n\geq 2,}
\]
where $\mu_A$ denotes the multiplication in $A$ and the second morphism in the second line is the projection from $A^{\text{!`}}\subseteq \mathrm T^{c,+}(V)$ onto $A$: it annihilates any term on the left-hand side of $A^{\text{!`}}\otimes A^{\text{!`}}$ of degree strictly smaller than $-1$ by the very definition of $A$ and $B$.
The Koszul differential from $\mathrm K^{-1}(A)$ to $\mathrm K^0(A)$ is induced by $\mu_A$.

The coassociativity of both $\Delta$ and $\Delta^+$ easily yields the compatibility between the Koszul differential and the right $B$-action.
\end{proof}
Slight modifications of the arguments in the proof of Proposition~\ref{p-kosz-bim} yield a similar statement for quadratic-linear algebras: in particular, this applies to $\mathrm U(\mathfrak g)$ and its Koszul complex $\mathrm{CE}(\mathfrak g)$.

\section{The $A_\infty$-bimodule structure over $K=\mathbb K$}\label{s-3}
For $\mathfrak g$ as in Section~\ref{s-0}, we consider the two commutative algebras $A=\mathrm S(\mathfrak g)$ and $B=\wedge(\mathfrak g^*)$.
Observe that $A$ is concentrated in degree $0$, while $B$ is non-negatively graded; in particular, both $A$ and $B$ may be regarded as flat $A_\infty$-algebras.

Set $K=\mathbb K$: the natural augmentation maps on $A$ and $B$ make $K$ into a left $A$-module and right $B$-module respectively.

According to~\cite{CFFR}, $K$ is further endowed with a non-trivial $A_\infty$-$A$-$B$-bimodule structure: its restriction on the left-hand and on the right-hand side yields the previous natural left $A$- and right $B$-action respectively: in other words, if we denote by $\mathrm d_K^{m,n}$, $m,n\geq 0$, the Taylor components of the $A_\infty$-$A$-$B$-bimodule structure on $K$, we have
\[
\mathrm d_K^{m,0}=0,\ m\neq 1,\ \mathrm d_K^{0,n}=0,\ n\neq 1,
\]
while $\mathrm d_K^{1,0}$ and $\mathrm d_K^{0,1}$ are induced by the left $A$- and right $B$-action.

Non-triviality means, on the other hand, that the Taylor components $\mathrm d_K^{m,n}$, for both $m$, $n$ non-trivial, is non-trivial, {\em e.g.}
\[
\mathrm m_K^{1,1}(a\otimes 1\otimes b)=\begin{cases}
\langle b,a \rangle,& a\in \mathfrak g,\ b\in B_1=\mathfrak g^*,\\
0,& \text{otherwise},
\end{cases}
\]
and $\langle\bullet,\bullet\rangle$ denotes the duality pairing between $\mathfrak g^*$ and $\mathfrak g$.

We refer to~\cite[Subsection 6.2]{CFFR} and in particular for the present situation to~\cite[Subsections 3.1-3.3]{CRT} for a detailed exposition of the $A_\infty$-$A$-$B$-structure on $K$.
We content ourselves here to recall the features which are relevant for later computations, referring to {\em loc.~cit.} for proofs, motivations and discussions.

\subsection{Admissible graphs, configuration spaces and superpropagators}\label{ss-3-1}
First of all, the Taylor components $\mathrm d_K^{m,n}$ of the $A_\infty$-$A$-$B$-bimodule structure on $K$ are described pictorially by the so-called admissible graphs, which we now describe in generality.

Let $Q^{+,+}$ denote the interior of the first quadrant in $\mathbb C$.
For a triple of non-negative integers $(n,k,l)$, such that $2n+k+l-1\geq 0$ (the meaning of the previous inequality will be clarified later on), an admissible graph $\Gamma$ of type $(n,k,l)$ is a directed graph, whose set of vertices lies in $Q^{+,+}\sqcup i\mathbb R^+\sqcup \mathbb R^+$: more precisely, $\Gamma$ admits $n$ distinct vertices of the first type in $Q^{+,+}$, $k$, resp.\ $l$, ordered vertices of the second type on $i\mathbb R^+$, resp.\ $\mathbb R^+$.
See Figure $1$ for an example of an admissible graph of type $(5,1,0)$.
Observe that the ordering of the $l$ vertices of the second type on $\mathbb R^+$ is natural, while the ordering of the $k$ vertices of the second type on $i\mathbb R^+$ is defined {\em via}
\[
i y_1<\cdots <iy_k\Leftrightarrow y_1>\cdots> y_k.
\]
Moreover, $\Gamma$ admits multiple edges and short loops, {\em i.e.} there may be more than one directed edge connecting two distinct vertices of $\Gamma$ (the direction of all such edges is the same) and there may be directed edges, whose starting point coincides with the endpoint, respectively.
In the present situation, as we will see later on, admissible graphs do not admit multiple edges but may admit short loops.
We denote by $\mathcal G_{n,k,l}$ the set of admissible graphs of type $(n,k,l)$; see also~\cite[Subsubsection 3.3.1]{CRT}.

W.r.t.\ the choice of a basis $\{x_i\}$ of $\mathfrak g$ as in Section~\ref{s-0}, we identify $A$ with $\mathbb K[x_i]$; similarly, we now identify $B$ with $\mathbb K[\partial_i]$, where now the partial derivatives $\{\partial_i\}$ w.r.t.\ the linear coordinates $\{x_i\}$ are assigned degree $1$ and anticommute with each other.

We denote by $\mathcal C_{2,0,0}^+$ the compactified configuration space of $2$ distinct points in $Q^{+,+}$ {\em modulo} rescalings (the I-cube): it is a compact, oriented smooth manifold with corners of dimension $3$.
We consider now the only non-trivial $4$-colored propagator $\omega^{+,-}$ on $\mathcal C_{2,0,0}^+$ in the present situation: namely, we consider the smooth $1$-form
\[
\omega^{+,-}(z_1,z_2)=\frac{1}{2\pi}\left[\mathrm d\ \mathrm{arg}(z_1-z_2)+\mathrm d\ \mathrm{arg}(\overline z_1-z_2)-\mathrm d\ \mathrm{arg}(\overline z_1+z_2)-\mathrm d\ \mathrm{arg}(z_1+z_2)\right],
\] 
for $(z_1,z_2)$ a pair of distinct points in $Q^{+,+}$.

The combination of~\cite[Lemma 5.4]{CFFR} and~\cite[Proposition 3.1]{CRT} proves that $\omega^{+,-}$ extends to a smooth, closed $1$-form on the compactified configuration space $\mathcal C_{2,0,0}^+$: we refer to~\cite[Lemma 3.4]{CRT} for the full list of the boundary properties of $\omega^{+,-}$.
We may consider the identity morphism of $\mathfrak g$, which may be regarded as an element $\tau$ of $\mathfrak g^*\otimes \mathfrak g[1]$: w.r.t.\ the basis $\{x_i\}$, $\tau$ may be written as $\tau=\partial_{x_i}\otimes \iota_{\mathrm d x_i}$, where $\iota$ denotes contraction.
It is clear that $\tau$ extends to a graded biderivation on $T_\mathrm{poly}(X)=\mathrm S(\mathfrak g)\otimes \wedge(\mathfrak g^*)$, for $X=\mathfrak g^*$. 
Then, $\omega^{+,-}\otimes\tau$ is a smooth, closed $1$-form on $\mathcal C_{2,0,0}^+$ with values in the endomorphism of the graded vector space $T_\mathrm{poly}(X)$ of degree $-1$ (whence its total degree is $0$).

We denote by $\mathcal C_{1,0,0}^+$ the compactified configuration space of a single point in $Q^{+,+}$ {\em modulo} rescalings: it is a compact, oriented smooth manifold with corners of dimension $1$.
It admits a natural smooth, exact $1$-form $\mathrm d\eta$, which is the smooth extension to $\mathcal C_{1,0,0}^+$ of the exterior derivative of the normalized angle function $Q^{+,+}\ni z\overset{\eta}\mapsto \arg(z)/2\pi$: observe that $\mathrm d\eta$ vanishes on the two boundary strata of codimension $1$ of $\mathcal C_{1,0,0}^+$.

As $X=\mathfrak g^*$ is a vector space, we may consider the divergence operator $\mathrm{div}=\partial_i\iota_{\mathrm dx_i}$ on $T_\mathrm{poly}(X)$ of degree $-1$ w.r.t.\ the standard volume form on $X$.
It is then clear that $\eta\otimes\mathrm{div}$ is a smooth, exact $1$-form on $\mathcal C_{1,0,0}^+$ with values in the endomorphisms of $T_\mathrm{poly}(X)$ (it also has total degree $0$). 

More generally, for a triple of non-negative integers satisfying the same inequality as above, we denote by $\mathcal C_{n,k,l}^+$ the compactified configuration space of $n$ distinct points in $Q^{+,+}$, $k$, resp.\ $l$, ordered points in $i\mathbb R^+$, resp.\ $\mathbb R^+$, {\em modulo} rescalings: it is a compact, oriented smooth manifold with corners of dimension $2n+k+l-1$ (whence the inequality).

For an admissible graph $\Gamma$ of type $(n,k,l)$, we denote by $e$ a general oriented edge of $\Gamma$: $e$ may be denoted also by $e=(v_1,v_2)$, $v_1$ and $v_2$ being its initial point and endpoint respectively.
Observe that we allow $v_1=v_2$; of course, $v_i$ may be either of the first or second type.

Associated to a directed edge $e$ of an admissible graph $\Gamma$ of type $(n,k,l)$, there are either natural projections $\mathcal C_{n,k,l}^+\overset{\pi_e}\to \mathcal C_{2,0,0}^+$, if $e=(v_1,v_2)$, $v_1\neq v_2$, or $\mathcal C_{n,k,l}^+\overset{\pi_e}\to \mathcal C_{1,0,0}^+$, if $e=(v,v)$, which simply forget all points in $\mathcal C_{n,k,l}^+$ except the one(s) corresponding to the endpoint(s) of $e$.
As is clear from the definition of $\pi_e$ and of $\mathcal G_{n,k,l}$, the image of $\pi_e$ may be actually a boundary stratum (even of codimension $2$) of $\mathcal C_{2,0,0}^+$.

With the previous notation, we define the superpropagator $\omega_e^K$ associated to a general edge $e$ of an admissible graph $\Gamma$ of type $(n,k,l)$ as 
\begin{equation}\label{eq-sup-prop}
\omega^K_e=\begin{cases}
\pi_e^*(\omega^{+,-})\otimes \tau_e,& \text{if $e=(v_1,v_2)$, $v_1\neq v_2$}\\
\frac{1}2\pi_e^*(\mathrm d\eta)\otimes \mathrm{div}_v,& \text{if $e=(v,v)$},
\end{cases}
\end{equation}
where $\tau_e$, resp.\ $\mathrm{div}_v$, is the endomorphism of $T_\mathrm{poly}(X)^{\otimes(n+k+l)}$ acting as $\tau$ on the components of $T_\mathrm{poly}(X)^{\otimes(n+k+l)}$ corresponding to the endpoints of $e$ and as the identity elsewhere, resp.\ as $\mathrm{div}$ on the component of $T_\mathrm{poly}(X)^{\otimes(n+k+l)}$ corresponding to the vertex $v$ and as the identity elsewhere.
Therefore, for any directed edge $e$ of $\Gamma$ in $\mathcal G_{n,k,l}$, $\omega_e^K$ is a closed element of $\Omega^1(\mathcal C_{n,k,l}^+,\mathrm{End}(T_\mathrm{poly}(X)^{\otimes(n+k+l})))$ of total degree $0$.

\subsection{Explicit formul\ae\ for the $A_\infty$-$A$-$B$-bimodule structure over $K$}\label{ss-3-2}
With the help of the superpropagator~\eqref{eq-sup-prop}, we proceed to define the $A_\infty$-$A$-$B$-bimodule structure over $K$.

We first associate to $\Gamma$ in $\mathcal G_{n,k,l}^+$ the following operator:
\begin{equation}\label{eq-form-K}
\mathcal O_\Gamma^K=\mu_{n+k+l+1}\circ\int_{\mathcal C_{n,k,l}^+}\prod_{e\in\mathrm E(\Gamma)}\omega^K_e:T_\mathrm{poly}(X)^{\otimes(n+k+l+1)}\to \mathbb K,
\end{equation}
where $\mu_{n+k+l}:T_\mathrm{poly}(X)^{\otimes(n+k+l+1)}\cong A^{\otimes(n+k+l+1)}\otimes B^{\otimes(n+k+l+1})\to \mathbb K^{\otimes 2(n+k+l+1)}=\mathbb K$ is the tensor product of the augmentation morphisms on $A$ and $B$, and $\mathrm E(\Gamma)$ denotes the set of edges of $\Gamma$.
We observe that the ordering in the product on the right-hand side of~\eqref{eq-form-K} is not important, as the total degree of each factor is $0$. 

It is pretty obvious that the operator $\mathcal O_\Gamma$ vanishes, if $\Gamma$ contains multiple edges (any power of the operator-valued $1$-form $\omega^K_e$ of degree bigger than $1$ vanishes); no short loop can be attached to a vertex of the second type (the operator-valued $1$-form $\omega_e^K$ vanishes, for $e$ a short loop based on $i\mathbb R^+\sqcup \mathbb R^+$, because $\mathrm \d\eta$ vanishes on the boundary strata of $\mathcal C_{1,0,0}^+$).
Finally, $\mathcal O_\Gamma$ vanishes unless $|\mathrm E(\Gamma)|=2n+k+l-1$, because the dimension of $\mathcal C_{n,k,l}^+$ must be equal to the form degree, which in turn equals by definition the number of edges of $\Gamma$.

Observe that there are natural injections of graded algebras $A,B\hookrightarrow T_\mathrm{poly}(X)$.
Then, the (non-shifted) Taylor component $\mathrm m_K^{k,l}$ of the $A_\infty$-$A$-$B$-bimodule structure over $K$ is defined as the composition
\begin{equation}\label{eq-bimod-K}
\xymatrix{A^{\otimes k}\otimes K\otimes B^{\otimes l}\ar@{^{(}->}[r] & T_\mathrm{poly}(X)^{k+l+1}\ar^-{\sum_{\Gamma\in\mathcal G_{0,k,l}}\mathcal O_\Gamma^K}[rr] & & K}.
\end{equation}
Here, $K=\mathbb K$ is regarded also as a subspace of $T_\mathrm{poly}(X)$.
The corresponding shifted Taylor components are denoted by $\mathrm d_K^{k,l}$.
It has been proved in~\cite[Subsection 6.2]{CFFR} and~\cite[Subsection 2.1]{CFR} that $K$ is a strictly unital $A_\infty$-$A$-$B$-bimodule.

\subsection{The $A_\infty$-bimodule structure on the $A_\infty$ bar construction of $A$}\label{ss-3-3}
We consider the $A_\infty$-bimodule structure over $K$ specified by Formul\ae~\eqref{eq-bimod-K}.
Since $A$, $B$ are flat, and $A$, $K$ are strictly unital, Proposition~\ref{p-bar} implies that there is an $A_\infty$-quasi-isomorphism of $A_\infty$-$A$-$B$-bimodules from $A\underline\otimes_A K$ to $K$.
A direct computation implies that $A\underline\otimes_A K$ is a dg vector space concentrated in non-positive degrees; recalling~\eqref{eq-tayl-tens}, its $A_\infty$-$A$-$B$-bimodule structure is given by  
\begin{equation}\label{eq-bar-tayl}
\begin{aligned}
\mathrm d_{A\underline\otimes_A K}^{0,0}(a\otimes (\widetilde a_1|\cdots|\widetilde a_q)\otimes 1)&=s\!\left((a\widetilde a_1)\otimes (\widetilde a_2|\cdots|\widetilde a_q)\otimes 1+\sum_{i=1}^{q-1}(-1)^i a\otimes (\widetilde a_1|\cdots|\widetilde a_i\widetilde a_{i+1}|\cdots|\widetilde a_q)\otimes 1+\right.\\
&\phantom{=}\left.+(-1)^q a\otimes (\widetilde a_1|\cdots|\widetilde a_{q-1})\otimes \varepsilon(\widetilde a_q)\right),\\
\mathrm d_{A\underline\otimes_A K}^{1,0}(a_1|a\otimes (\widetilde a_1|\cdots|\widetilde a_q)\otimes 1)&=s\!\left((a a_1)\otimes (\widetilde a_1|\cdots|\widetilde a_q)\otimes 1\right),\\
\mathrm d_{A\underline\otimes_A K}^{0,n}(a\otimes (\widetilde a_1|\cdots|\widetilde a_q)\otimes 1|b_1|\cdots|b_n)&=(-1)^q\sum_{l=0}^q s\!\left(a_1\otimes (\widetilde a_1|\cdots|\widetilde a_l)\otimes s^{-1}(\mathrm d_K^{q-l,n}(\widetilde a_{l+1}|\cdots|\widetilde a_q|1|b_1|\cdots|b_n))\right),
\end{aligned}
\end{equation}
and, in all other cases, the Taylor components are trivial; $\varepsilon$ denotes here the augmentation map of $A$.

The first identity in~\eqref{eq-bar-tayl} yields the identification between the dg vector space $\left(A\underline\otimes_A K,\mathrm d_{A\underline\otimes_A K}^{0,0}\right)$ with the actual bar complex of the left augmentation module $K=\mathbb K$ over $A$, whence the name.

Further, the identities~\eqref{eq-bar-tayl} imply that the left $A_\infty$-$A$-module structure on the bar complex $A\underline\otimes _A K$ of $K$ is the standard one, while the non-triviality of the $A_\infty$-$A$-$B$-bimodule structure on $K$ yields non-triviality of the right $A_\infty$-$B$-module structure on $A\underline\otimes_A K$.

\section{Deformation quantization of the $A_\infty$-$A$-$B$-bimodule $K$}\label{s-4}
We borrow notation from Section~\ref{s-3}.
Let $\hbar$ be a formal parameter as in Section~\ref{s-0} and consider the $\hbar$-shifted Kirillov--Kostant--Souriau Poisson bivector field $\hbar\pi$.

\subsection{Deformation quantization of $A$}\label{ss-4-1}
The formality $L_\infty$-quasi isomorphism $\mathcal U_A$ of Kontsevich~\cite{K} yields an associative algebra structure over $A_{\hbar}=A[\![\hbar]\!]$, which we denote by $(A_{\hbar},\star_{\hbar})$.
More precisely, if $\mathrm m_A$ denotes the standard commutative, associative product on $A$, extended by $\hbar$-bilinearity to $A_{\hbar}$, the $\hbar$-formal bidifferential operator $\mathrm m_{A_{\hbar}}^2=\mathrm m_A+\mathcal U_A(\hbar\pi)$ defines an associative product $\star_{\hbar}$ over $A_{\hbar}$, and the properties of $\mathcal U$ imply 
\[
a_1\star_{\hbar} a_2-a_2\star_{\hbar} a_1=\hbar[a_1,a_2],\ a_i\in \mathfrak g\subseteq A_{\hbar}.
\] 

Observe that we may safely set $\hbar=1$, see~\cite[Theorem 8.3.1]{K} for an explanation thereof.
Accordingly, we may thus consider the associative algebra $(A,\star_A)$.

We recall briefly the Duflo element $\sqrt J$ of $\mathfrak g$.
By definition, $\sqrt J$ is an invertible, $\mathfrak g$-invariant differential operator on $A$ with constant coefficients of infinite order acting on $A$: it is a formal linear combination with rational coefficients of traces of powers of the adjoint representation of $\mathfrak g$ on itself.
We just recall that the rational coefficients of $\sqrt J$ are the (modified) Bernoulli numbers with generating function $\sqrt{(1-e^{-x})/x}$.
\begin{Prop}\label{p-DK}
For $\mathfrak g$ as in Section~\ref{s-0}, there exists an explicit algebra isomorphism from $(A,\star)$ to $(\mathrm U(\mathfrak g),\cdot)$, given by composition of the Duflo element $\sqrt J$ with the symmetrization (or PBW, short for Poincar\'e--Birkhoff--Witt) isomorphism from $A$ to $\mathrm U(\mathfrak g)$ (as vector spaces). 
\end{Prop}
For a detailed proof, we refer to~\cite[Subsection 3.2]{CFR}: actually, Proposition~\ref{p-DK} has been proved elsewhere~\cite[Subsection 8.3]{K} and~\cite{Sh3}, but the techniques are quite different and moreover the modified Duflo element actually appears.
The latter fact does not cause any problem, see~\ref{ss-4-3} for a detailed explanation.

\subsection{Deformation quantization of $B$}\label{ss-4-2}
Set $X=\mathfrak g^*$, $\widehat X=\mathfrak g[1]$, both viewed now as graded linear manifolds: then, $A=\mathcal O(X)$ and $B=\mathcal O(\widehat X)$.
Furthermore, $T_\mathrm{poly}(X)=T_\mathrm{poly}(\widehat X)$, thus $\pi$ may be regarded as a quadratic vector field $\widehat \pi$ on $\widehat X$, which squares to $0$: viewed as a derivation of $B$, it obviously coincides with the Chevalley--Eilenberg differential $\mathrm d$ on the Chevalley--Eilenberg complex of $\mathfrak g$ with values in the trivial $\mathfrak g$-module $\mathbb K$.

Let use consider the $\hbar$-shifted vector field $\hbar\widehat\pi$ on $\widehat X$.
Then, the formality $L_\infty$-quasi isomorphism $\mathcal U_B$~\cite{CF} yields a flat $A_\infty$-structure on $B_{\hbar}=B[\![\hbar]\!]$.
More precisely, if $\mathrm m_B$ denotes the standard commutative, associative product on $B$, extended by $\hbar$-bilinearity to $B_{\hbar}$, the $\hbar$-formal multidifferential operator $\mathrm m_{B_{\hbar}}^2=\mathrm m_B+\mathcal U_B(\hbar\widehat\pi)$ defines an $A_\infty$-structure on $B_{\hbar}$: since $\hbar\widehat\pi$ is a vector field, the results of~\cite[Subsubsubsection 7.3.1.1]{K} imply that $\mathcal U_B(\hbar\widehat\pi)=\hbar\widehat\pi$, hence $B_{\hbar}$ is a dg algebra with $\mathrm m_{B_{\hbar}}^1=\hbar\mathrm d$ and wedge product $\mathrm m_{B_{\hbar}}^2=\mathrm  m_B$.

We may thus safely set $\hbar=1$ and consider the deformation quantization of $B$ as the Chevalley--Eilenberg complex of $\mathfrak g$ with values in the trivial module $\mathbb K$.

\subsection{Deformation quantization of $K$}\label{ss-4-3}
We now consider the $A_\infty$-$A$-$B$-bimodule structure over $K$ from Subsection~\ref{ss-3-2}; according to the previous subsections, we consider the deformed algebras $(A_{\hbar},\star_{\hbar})$ and $(B_{\hbar},\mathrm d_{B_{\hbar}})$ (where $\mathrm d_{B_{\hbar}}$ denotes collectively the $A_\infty$-algebra structure on $B_{\hbar}$).
Both $A_{\hbar}$ and $B_{\hbar}$ may be regarded as flat $A_\infty$-algebras: we are interested in the corresponding deformation quantization of the $A_\infty$-$A$-$B$-bimodule structure over $K$.

According to~\cite[Theorem 7.2]{CFFR} and~\cite[Theorem 3.5]{CRT}, such a deformation quantization is yielded by a formality $L_\infty$-quasi-isomorphism $\mathcal U$ in presence of two branes (here, the two branes (or coisotropic submanifolds) of $X$ are $X$ itself and $\{0\}$).
Then, $\mathrm m_{K_{\hbar}}=\mathrm m_K+\mathcal U(\hbar\pi)$ defines an $A_\infty$-$A_{\hbar}$-$B_{\hbar}$-bimodule structure over $K_{\hbar}=K[\![\hbar]\!]$.

More precisely, the (non-shifted) Taylor components $\mathrm m_{K_{\hbar}}^{k,l}$ of the $A_\infty$-$A_{\hbar}$-$B_{\hbar}$-bimodule structure on $K_{\hbar}$ are defined as the composed maps
\[
\begin{aligned}
\xymatrix{A_{\hbar}^{\otimes_{\hbar} k}\otimes_{\hbar} K_{\hbar}\otimes_{\hbar} B_{\hbar}^{\otimes_{\hbar} l}\cong(A^{\otimes k}\otimes K\otimes B^{\otimes l})[\![\hbar]\!]\ar@{^{(}->}[r] & T_\mathrm{poly}(X)^{\otimes(k+l+1)}[\![\hbar]\!]\ar[rrrrr]^-{\sum_{n\geq 0}\frac{1}{n!}\sum_{\Gamma\in\mathcal G_{n,k,l}}\mathcal O_\Gamma^K(\underset{n}{\underbrace{\hbar\pi,\dots,\hbar\pi}})} & & & & & K_{\hbar}}, 
\end{aligned}
\]
where all maps are extended $\hbar$-linearly, and we borrowed notation from Subsections~\ref{ss-3-1},~\ref{ss-3-2} and~\ref{ss-3-3}.
The corresponding shifted Taylor components are denoted by $\mathrm d_{K_{\hbar}}^{k,l}$.

We observe that arguments analogous to the ones in~\cite[Subsubsection 8.3.1]{K} imply that $\hbar$ may be safely set to be $1$: by abuse of notation, we denote by $\mathrm d_K^{k,l}$ and $\mathrm m_K^{k,l}$ the shifted and non-shifted Taylor components of the deformed $A_\infty$-bimodule structure over $K$: there is no risk of confusion, because from now on we will consider only the deformed $A_\infty$-bimodule structure.
\begin{Prop}\label{p-K-bimod}
The Taylor components $\mathrm m_K^{k,0}$ and $\mathrm m_K^{0,l}$ are trivial unless $k=l=1$; in the first case, the component $\mathrm m_K^{1,0}$ equals the augmentation map of $A$ composed with the Duflo element $\sqrt J$, while $\mathrm m_K^{0,1}$ equals simply the augmentation map of $B$. 
\end{Prop}
\begin{proof}
Since $A$ and $K$ are both concentrated in degree $0$, necessarily $\mathrm m_K^{k,0}$ is trivial unless $k=1$.

Let us therefore consider $\mathrm m_K^{0,l}$, for $l\geq 1$.
Degree reasons imply that $\mathrm m_K^{0,l}(1,b_1,\dots,b_l)$ is non-trivial only if $\sum_{i=1}^l|b_i|=l-1$.
If $\Gamma$ is an admissible graph of type $(n,0,l)$, then we claim
\[
\mathcal O_\Gamma^K(\underset{n}{\underbrace{\pi,\dots,\pi}},1,b_1,\dots,b_l)=0, 
\]
unless $\Gamma$ is the only non-trivial element of $\mathcal G_{0,0,1}$.
This also proves the claim that $\mathrm m_K^{0,1}$ is the augmentation map of $B$.

As $l\geq 1$, we claim that $|b_i|\geq 1$, for all $i=1,\dots,l$.
Otherwise, $\Gamma$ would have a vertex $1\leq i\leq l$ of the second type of valence $0$ ({\em i.e.} no edge would depart from or arrive to the said vertex).
Assume first that $l\geq 2$ and that {\em e.g.} the first vertex of the second type is $0$-valent: using rescalings on $\mathcal C_{n,0,l}$, we fix the second vertex of the first type to $1$.
Then, there is nothing to be integrated over the interval $(0,1)$, hence $\mathcal O_\Gamma^K$ is trivial.
If $l=1$, assume $n\geq 1$ and the only vertex of the second type to be $0$-valent: using once again rescalings on $\mathcal G_{n,0,1}$, we may fix on the unit circle one of the $n$ vertices of the first type.
Then, there is also nothing to be integrated over $\mathbb R^+$ and therefore $\mathcal O_\Gamma^K$ vanishes. 

The computation of $\mathrm m_K^{1,0}$ has been performed in~\cite{W}.
The only difference between the claim above and~\cite[Proposition 17, Appendix B]{W} is that the augmentation map is composed here with the actual Duflo element, while in~\cite[Proposition 17, Appendix B]{W} it is composed with the modified Duflo element.

The claim lies in computing explicitly $\mathrm m_K^{1,0}$.
Let $\Gamma$ be an admissible graph of type $(n,1,0)$, for $n\geq 0$: observe that, by construction of $\mathcal O_\Gamma^K$, from a general vertex of the first type depart exactly two edges and to it arrives at most one edge because of the linearity of $\pi$.
Denote by $p$ the number of edges hitting the only vertex of the second type on $i\mathbb R^+$: as $\Gamma$ admits no multiple edges, $p\leq n$.
On the other hand, the polynomial degree in $A$ of $\mathcal O_\Gamma^K$ equals $n-(2n-p)=-n+p\geq 0$: namely, each vertex of the first type carries a copy of the linear bivector field $\pi$, and to every edge is associated by construction a derivative, thus $2n-p$ equals precisely the number of edges hitting vertices of the first type.
Therefore, $p=n$: this means that from each vertex of the first type depart two edges, of which exactly one hits a vertex of the first type and the other one hits the only vertex of the second type.

In other words, $\Gamma$ is a disjoint union of wheel graphs $W_n$, $n\geq 1$; observe that admissible graphs here admit short loops.
Pictorially, such a graph $\Gamma$ is of the form
\bigskip
\begin{center}
\resizebox{0.2 \textwidth}{!}{\input{wheel_augm.pstex_t}}\\
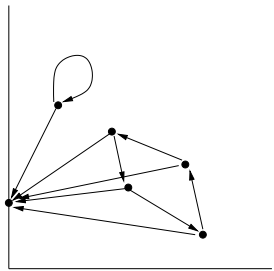
\text{Figure 1 - The disjoint union $\Gamma$ of the $1$-wheel $W_1$ and the $4$-wheel $W_4$} \\
\end{center}
\bigskip
The arguments in the proof of~\cite[Theorem 18, Appendix B]{W} can be borrowed {\em verbatim}, but we have to keep track of the additional differential operator associated to the $1$-wheel $W_1$.
Slightly adapting the arguments of~\cite[Subsubsection 4.1.1]{CRT} to the present situation, the operator $\mathcal O_{W_1}^K$ equals precisely $1/4 c_1$, $c_1$ being the trace of the adjoint representation of $\mathfrak g$.
Therefore, $\mathrm m_K^{1,0}$ equals the augmentation map of $A$ composed with the Duflo element $\sqrt J$: the action of the automorphism $\exp(1/4 c_1)$ on the modified Duflo element (which corresponds to the total contribution coming from even wheel graphs), computed in~\cite[Proposition 17, Appendix B]{W}, yields the actual Duflo element $\sqrt J$.
\end{proof}
Observe first that the results of Proposition~\ref{p-K-bimod} are coherent with Proposition~\ref{p-DK}.

The presence of short loops in the formality theorem in presence of two branes is relatively new: short loops did not appear neither in the seminal paper~\cite{CF}, nor in~\cite{CT}, where formality in presence of two branes has been applied to problems in Lie algebra theory.
They have been first introduced to correct a slight problem (until then passed unnoticed) arising from a regular boundary contribution to the $4$-colored superpropagators on the boundary stratum $\mathcal C_2\times \mathcal C_{1,0,0}^+$ of $\mathcal C_{2,0,0}^+$, see~\cite[Subsection 7.2]{CFFR} and~\cite[Section 1]{CRT} for more details.

The following technical Lemma has already been used in~\cite[Subsection 3.2]{CFR} without proof: we present here a concise proof, and we comment it afterward.
\begin{Lem}\label{l-tr-star}
The trace of the adjoint representation $c_1$ of a finite-dimensional Lie algebra $\mathfrak g$ over $\mathbb K$ is a derivation of the deformed product $\star_A$ on $A$, {\em i.e.} $c_1(a_1\star_A a_2)=c_1(a_1)\star_A a_2+a_1\star c_1(a_2)$, $a_i$ in $A$.
\end{Lem}
\begin{proof}
Recall from~\cite{K} that the $L_\infty$-quasi-isomorphism $\mathcal U_A$ maps MC (shortly for Maurer--Cartan) elements of $T_\mathrm{poly}(X)$ to MC elements of $D_\mathrm{poly}(X)$, the Hochschild subcomplex of multidifferential operators on $A$.
We further consider the $1$-dimensional Grassmann algebra $\mathbb K[\epsilon]$ and accordingly $T_\mathrm{poly}(X)[\epsilon]$, $D_\mathrm{poly}(X)[\epsilon]$ with the induced dg Lie algebra structures and the $\epsilon$-linear extension $\mathcal U_A$.

Since $c_1$ is $\mathfrak g$-invariant, $\pi+\epsilon c_1$ is a Maurer--Cartan element of $T_\mathrm{poly}(X)[\epsilon]$, hence $\mathrm m_A+\mathcal U_A(\pi+\epsilon c_1)$ is also a MC element of $D_\mathrm{poly}(X)[\epsilon]$.
A general MC element $\gamma$ of $D_\mathrm{poly}(X)$ splits as $\gamma_1+\epsilon \gamma_0$, $\gamma_1$ a bidifferential operator and $\gamma_0$ a differential operator: the corresponding MC equation implies that $i)$ $\mathrm m_A+\gamma_1$ defines an associative product on $A$ and $ii)$ $\gamma_0$ is a derivation for the product $\mathrm m_A+\gamma_1$.

It suffices now to evaluate $\mathcal U_A(\pi+\epsilon c_1)$: it splits as $\mathcal U_A(\pi)+\epsilon c_1$, because the higher Taylor components of $\mathcal U_A$ vanish if at least one of their arguments is an affine vector field, see~\cite[Subsubsubsection 7.3.3.1]{K}.   
\end{proof}
Lemma~\ref{l-tr-star} implies, in particular, that $\exp(1/4 c_1)$ is an automorphism of $(A,\star_A)$, therefore there is no contradiction actually between~\cite[Proposition 17, Appendix B]{W} and Proposition~\ref{p-K-bimod}, we may choose either the actual Duflo element $\sqrt J$ or its modified version.

\subsection{Relationship with the Chevalley--Eilenberg complex $\mathrm{CE}(\mathfrak g)$}\label{ss-4-4}
Borrowing notation from the previous Subsections and Section~\ref{s-2}, we consider the graded vector space $A\otimes B^*=\mathrm S(\mathfrak g)\otimes\wedge(\mathfrak g)\cong A[\theta_i]$: we endow it with a dg bimodule structure {\em via}
\begin{align*}
\mathrm d(a(\theta_{i_1}\cdots \theta_{i_q}))&=\sum_{k=1}^q(-1)^{k+1}(a\star_Ax_{i_k})(\theta_{i_1}\cdots \widehat{\theta_{i_k}}\cdots \theta_{i_q})+\sum_{1\leq k<l\leq q}(-1)^{k+l}f_{i_k,i_l}^j a(\theta_j\theta_{i_1}\cdots\widehat{\theta_{i_k}}\cdots\widehat{\theta_{i_l}}\cdots \theta_{i_q}),\\
\mathrm m_L(a_1\otimes(a\theta^I))&=(a_1\star_A a)\theta^I,\\ 
\mathrm m_R((a\theta^I)\otimes b_1)&=(-1)^{|I||b_1|}a\otimes b_1(\theta^I),
\end{align*}
where we have used the identification $B=\mathbb K[\partial_{\theta_i}]$.

We further consider the (twisted) augmentation map $\varepsilon\circ\sqrt J$ from $(A,\star_A)$ to $\mathbb K$.
\begin{Prop}\label{p-CE-def}
For $\mathfrak g$ as in Section~\ref{s-0}, there is a commutative square of dg bimodules
\[
\xymatrix{A\otimes B^*\ar[d]_-{(\mathrm{PBW}\circ\sqrt J)\otimes 1}^-{\cong}\ar@{->>}[rr]_-{\sim}^-{\varepsilon\circ(\sqrt J\otimes 1)}& & \mathbb K\ar@{=}[d]\\
\mathrm{CE}(\mathfrak g)\ar@{->>}[rr]^-\sim_-\varepsilon & &\mathbb K},
\]
where the dg bimodule structure on the upper, resp.\ lower, row is described above, resp.\ in Section~\ref{s-2}; the horizontal arrows are quasi-isomorphisms of dg vector spaces, and the vertical arrows are actual isomorphisms of dg bimodules, $\mathrm{PBW}$ denotes the PBW isomorphism of vector spaces and $\sqrt J$ is the Duflo element.
\end{Prop}
\begin{proof}
The proof follows immediately from the definition of the dg bimodule structures on $\mathrm{CE}(\mathfrak g)$ and $A\otimes B^*$ and Propositions~\ref{p-DK} and~\ref{p-K-bimod}. 
\end{proof}
Therefore, we may rightfully call the $4$-tuple $(A\otimes B^*,\mathrm d,\mathrm m_L,\mathrm m_R)$ the Chevalley--Eilenberg complex of $\mathfrak g$.

\section{An explicit $A_\infty$-quasi-isomorphism between the deformed $A_\infty$ bar complex and the Chevalley--Eilenberg complex of $A$}\label{s-5}
We consider the deformed flat $A_\infty$-algebras $A$, $B$ and the deformed flat $A_\infty$-$A$-$B$-bimodule $K$ from Section~\ref{s-4}, and the corresponding $A_\infty$-$A$-$B$-bimodule $A\underline{\otimes}_A K$: degree reasons and the results of Section~\ref{s-4} imply that Formul\ae~\eqref{eq-bar-tayl} must be modified as  
\begin{equation}\label{eq-bar-tayl-def}
\begin{aligned}
\mathrm d_{A\underline\otimes_A K}^{0,0}(a\otimes (\widetilde a_1|\cdots|\widetilde a_q)\otimes 1)&=s\!\left((a\star_A \widetilde a_1)\otimes (\widetilde a_2|\cdots|\widetilde a_q)\otimes 1+\sum_{i=1}^{q-1}(-1)^i a\otimes (\widetilde a_1|\cdots|\widetilde a_i\star_A \widetilde a_{i+1}|\cdots|\widetilde a_q)\otimes 1+\right.\\
&\phantom{=}\left.+(-1)^q a\otimes (\widetilde a_1|\cdots|\widetilde a_{q-1})\otimes \varepsilon(\sqrt J(\widetilde a_q))\right),\\
\mathrm d_{A\underline\otimes_A K}^{1,0}(a_1|a\otimes (\widetilde a_1|\cdots|\widetilde a_q)\otimes 1)&=s\!\left((a\star_A a_1)\otimes (\widetilde a_1|\cdots|\widetilde a_q)\otimes 1\right),\\
\mathrm d_{A\underline\otimes_A K}^{0,n}(a\otimes (\widetilde a_1|\cdots|\widetilde a_q)\otimes 1|b_1|\cdots|b_n)&=(-1)^q\sum_{l=0}^q s\!\left(a_1\otimes (\widetilde a_1|\cdots|\widetilde a_l)\otimes s^{-1}(\mathrm d_K^{q-l,n}(\widetilde a_{l+1}|\cdots|\widetilde a_q|1|b_1|\cdots|b_n))\right), 
\end{aligned}
\end{equation}
borrowing notation from Section~\ref{s-4}.
(The remaining Taylor components are, as before, trivial.)

Since $A\underline\otimes_A K$ and $A\otimes B^*$ are both resolutions of the (left) augmentation module $K$ over $A$, arguments from abstract homological algebra imply that they are quasi-isomorphic to each other as complexes of free left $A$-modules.
More precisely, there is a natural quasi-isomorphism from $A\otimes B^*$ to $A\underline\otimes_A K$ as complexes of left $A$-modules
\begin{equation}\label{eq-skew}
\Phi(\theta_{i_1}\cdots\theta_{i_q})=\sum_{\sigma\in\mathfrak S_q}(-1)^\sigma 1\otimes (x_{\sigma(i_1)}|\cdots|x_{\sigma(i_q)})\otimes 1,\quad 1\leq i_1<\cdots< i_q\leq d.
\end{equation}
It suffices to define the morphism $\Phi$ on monomials of the form $\theta_{i_1}\cdots\theta_{i_q}$, and then extend it $A$-linearly on the left.
It follows immediately that $\Phi$ is of degree $0$ and commutes with left $A$-action.

An easy computation shows that $\Phi$ commutes with differentials. 
It is a quasi-isomorphism since $\Phi(1)=1$.

\begin{Thm}\label{t-koszul-bar-1}
For $\mathfrak g$ as in Section~\ref{s-0}, the morphism~\eqref{eq-skew} defines a strict $A_\infty$-quasi-isomorphism from $A\otimes B^*$ to $A\underline\otimes_A K$, where the $A_\infty$-$A$-$B$-bimodule structures on $A\otimes B^*$ and $A\underline\otimes_A K$ are described in Subsection~\ref{ss-4-4}.
\end{Thm}
\begin{proof}
We know that~\eqref{eq-skew} is a morphism of degree $0$ from $A^*\otimes B$ to $A\underline\otimes_A K$: we declare (the conjugation w.r.t.\ $s$ of) $\Phi$ to be the $(0,0)$-th Taylor component of the desired $A_\infty$-quasi-isomorphism, while for $(m,n)$ such that $m+n\geq 1$, we set simply $0$.
In other words, $\Phi$ is a strict $A_\infty$-morphism between $A_\infty$-$A$-$B$-bimodules.

As $\Phi$ is a strict $A_\infty$-morphism, and recalling that $A\otimes B^*$ is a dg bimodule and the Taylor components~\eqref{eq-bar-tayl-def}, the only non-trivial identities to be checked are
\begin{align}
\label{eq-A_inf-1} \mathrm d_{A\underline\otimes_A K}^{0,1}(\Phi(\eta)|b_1)&=\Phi(\mathrm d_{A\otimes B^*}^{0,1}(\eta|b_1)),\\
\label{eq-A_inf-2} \mathrm d_{A\underline\otimes_A K}^{0,p}(\Phi(\eta)|b_1|\cdots|b_p)&=0,\quad p\geq 2,
\end{align}
for $b_i$, $i=1,\dots,p$, resp.\ $\eta$, a general element of $B$, resp.\ $A\otimes B^*$.
Here, we have denoted by $\mathrm d_{A\otimes B^*}^{0,1}$ the (shifted) Taylor component of the $A_\infty$-bimodule structure on $A\otimes B^*$ corresponding to $\mathrm m_R$.

We prove first Identity~\eqref{eq-A_inf-2}: $A$-linearity implies that we may take $\eta$ of the form $\theta_{i_1}\cdots\theta_{i_q}$, $1\leq i_1<\cdots<i_q\leq d$.
Recalling now Identity~\eqref{eq-skew},
we rewrite the left-hand side in~\eqref{eq-A_inf-2} as
\[
\mathrm d_{A\underline\otimes_A K}^{0,p}(\Phi(\eta)|b_1|\cdots|b_p)=(-1)^q\sum_{l=0}^q\sum_{\sigma\in \mathfrak S_q}(-1)^\sigma s\!\left(1\otimes (x_{\sigma(i_1)}|\cdots|x_{\sigma(i_l)})\otimes s^{-1}(\mathrm d_K^{q-l,p}(x_{\sigma(i_{l+1})}|\cdots|x_{\sigma(i_q)}|1|b_1|\cdots|b_p)\right).
\]
We now analyze the last factor on the right-hand side: first of all, degree reasons imply that, for $0\leq l\leq q$,
\[
-(q-l)-1+\sum_{j=1}^p(|b_j|-1)+1\overset{!}=-1\Longleftrightarrow \sum_{j=1}^p |b_j|=p+q-l-1.
\]

We consider an admissible graph $\Gamma$ of type $(n,q-l,p)$, $n\geq 0$ and $l$, $q$, $p$ as before.

We first assume $n\geq 1$.
By construction, from each vertex of the first type of $\Gamma$ depart exactly two edges and to it arrive at most one edge; the $q-l$ vertices of the second type on $i\mathbb R^+$ can be only endpoints of edges, while the vertices of the second type on $\mathbb R^+$ can be only starting points of edges of $\Gamma$.
The arguments in the proof of the first statement of Proposition~\ref{p-K-bimod} imply that from each vertex of the second type on $\mathbb R^+$ of $\Gamma$ must depart at least one edge.
Similar arguments imply that every vertex of the second type on $i\mathbb R^+$ of $\Gamma$ is the endpoint of exactly one edge (observe that the arguments in $A$ are of polynomial degree $1$).
We now consider the polynomial degree of the corresponding multidifferential operator acting on $\{x_{\sigma(i_{l+1})},\dots,x_{\sigma(i_q)}\}$: it is precisely 
\[
n-\left(2n+\sum_{i=1}^p|b_i|-(q-l)\right)=-n-\sum_{i=1}^p|b_i|+q-l=-n-p+1.
\]
Observe that the first copy of $n$ appears because there are $n$ copies of a linear bivector field; to each edge is associated a derivative, and the number of derivatives acting on the copies of $\pi$ equals $2n$ plus the sum of the degrees of the elements $b_i$ of $B$ minus the number of arrows hitting the $q-l$ vertices of the second type on $i\mathbb R^+$.
The latter number is precisely $q-l$ by the previous arguments. 
As $n\geq 0$ and $p\geq 2$, the inequality $-n-p+1\geq 0$ is never satisfied, whence the claim.

If $n=0$, we may borrow the arguments of the proof of Identity (11) of~\cite[Theorem 5.1]{FRW}, and the claim follows. 

It remains to prove Identity~\eqref{eq-A_inf-1}.
We first evaluate the right-hand side: using the arguments of Section~\ref{s-2}, we may write $b_1=\partial_{j_1}\cdots\partial_{j_p}$, whence the right-hand side takes the form $(-1)^{(p+1)q}\Phi(b_1(\eta))$, where $\partial_i=\partial_{\theta_i}$.

The left-hand side of Identity~\eqref{eq-A_inf-1} has the explicit form
\[
\begin{aligned}
\mathrm d_{A\underline\otimes_A K}^{0,1}(\Phi(\eta)|b_1)&=(-1)^q\sum_{l=0}^q\sum_{\sigma\in \mathfrak S_q}(-1)^\sigma s\!\left(1\otimes (x_{\sigma(i_1)}|\cdots|x_{\sigma(i_l)})\otimes s^{-1}(\mathrm d_K^{q-l,1}(x_{\sigma(i_{l+1})}|\cdots|x_{\sigma(i_q)}|1|b_1)\right)=\\
&=(-1)^q\sum_{\sigma\in \mathfrak S_q}(-1)^\sigma s\!\left(1\otimes (x_{\sigma(i_1)}|\cdots|x_{\sigma(i_{q-p})})\otimes s^{-1}(\mathrm d_K^{p,1}(x_{\sigma(i_{q-p+1})}|\cdots|x_{\sigma(i_q)}|1|b_1)\right),
\end{aligned}
\]
where the second equality follows because of degree reasons.

First, if $q\leq p-1$, by degree reasons both sides of Identity~\eqref{eq-A_inf-1} vanish: this follows immediately from the previous formul\ae.
It remains therefore to prove the claim in the case $q\leq p$.

We consider an admissible graph $\Gamma$ of type $(n,p,1)$ for $n\geq 1$.
Repeating {\em verbatim} the arguments in the previous paragraph with obvious due changes, we find that the polynomial degree of the corresponding multidifferential operator acting on $\{x_{\sigma(i_{q-p+1})},\dots,x_{\sigma(i_q)}\}$ is precisely $-n$: as $n\geq 1$, the inequality is never satisfied.

This forces $n=0$, therefore by its very construction, $\mathrm d_K^{p,1}(x_{\sigma(i_{q-p+1})}|\cdots|x_{\sigma(i_q)}|1|b_1)$ equals the non-deformed corresponding expression.
This has been in turn computed explicitly in the second part of the proof of~\cite[Theorem 5.1]{FRW}, to which we refer for details.

Hence, the claim follows.
\end{proof}

\end{document}

%% file: wheel_augm.pstex_t
\begin{picture}(0,0)%
\includegraphics{wheel_augm.pstex}%
\end{picture}%
\setlength{\unitlength}{4144sp}%
\begingroup\makeatletter\ifx\SetFigFont\undefined%
\gdef\SetFigFont#1#2#3#4#5{%
  \reset@font\fontsize{#1}{#2pt}%
  \fontfamily{#3}\fontseries{#4}\fontshape{#5}%
  \selectfont}%
\fi\endgroup%
\begin{picture}(3667,3624)(2646,-6823)
\end{picture}%